\documentclass[12pt]{article}
\usepackage{amsmath,amsthm,amssymb}
\usepackage{stmaryrd}
\SetSymbolFont{stmry}{bold}{U}{stmry}{m}{n}
\usepackage{amsfonts}
\usepackage{mathrsfs}
\usepackage{latexsym,bm}
\usepackage{yfonts}
\usepackage[all,poly,knot]{xy}
\usepackage{tipa}
\usepackage{fancyhdr}
\usepackage{indentfirst}
\usepackage{cite}
\usepackage[colorlinks=true,citecolor=blue,linkcolor=blue,anchorcolor=green,urlcolor=magenta,breaklinks=true,unicode]{hyperref}

\newtheorem{theorem}{Theorem}[section]
\newtheorem{Legendre theorem}{Legendre Theorem}[section]
\newtheorem{lemma}[theorem]{Lemma}
\newtheorem{corollary}[theorem]{Corollary}

\newtheorem{proposition}[theorem]{Proposition}

\begin{document}

\begin{center}
{\bf \large   A note on the zero divisor graph of the ring of Lipschitz integers modulo $n$ }

{ Hengbin Zhang\footnote{Corresponding author. E-mail address: zhhb210@163.com.}}\\
{\it  College of Mathematical Science, Yangzhou University, Yangzhou, P.R.China}\\
\end{center}

\noindent{\bf Abstract.} In a recent paper, Grau et al. (2017) studied the zero divisor graphs of the ring of Lipschitz integers modulo $n$, and computed the domination number of the undirected zero divisor graph of the ring of Lipschitz integers modulo $n$. But the case $n$ is a power of prime numbers remained open. In this note, this problem is solved. We also show the automorphism group of the zero divisor graph of the ring of Lipschitz integers modulo $2^{s}$.



\section{Introduction}

The zero divisor graph of the ring is a graph that the vertices are the elements of the ring and two distinct vertices $a$ and $b$ are adjacent if and only if $ab=0$. This concept was first introduced by Beck in \cite{B88}. Anderson and Livingston \cite{A99} reduced the vertex set to the set of non-zero zero divisors of the ring, and they also studied the automorphism group of the zero divisor graph of a  finite commutative ring. Later, Redmond \cite{R02} extended the definition of the zero divisor graph to the directed zero divisor graph and the undirected zero divisor graph. For recent study on the zero divisor graphs, see the papers \cite{A04, A06, A07, B09, M16, W14, W15, J17, Z16, Z17}.

We use $Z^{*}(R)$ and $U(R)$ to denote the set of non-zero zero divisors and the group of units of a ring $R$, respectively. If $a\in R$, then the annihilator of $a$ is $ann(a)=\{b\in R| ab=0 \}$. For a set $S$, $|S|$ denotes the size of $S$, $S\setminus T$ denotes the set of elements that belong to $S$ and not to set $T$.

Let $R$ be a ring with identity. The directed zero divisor graph $\Gamma(R)$ is defined with vertices  $Z^{*}(R)$, where $a\rightarrow b$ is an edge between distinct vertices $a$ and $b$ if and only if $ab=0$. The undirected zero divisor graph $\overline{\Gamma}(R)$ is defined with vertices  $Z^{*}(R)$, where distinct vertices $a$ and $b$ are adjacent if and only if $ab=0$ or $ba=0$.

The ring $\mathbb{Z}[i, j, k]$ of Lipschitz integer quaternions is extended by the Gaussian integers. Let $\mathbb{Z}_{n}[i, j, k] := \{a + bi + cj + dk \mid a, b, c, d \in \mathbb{Z}_{n}\}$ be a factor ring of $\mathbb{Z}[i, j, k]$, which is called the ring of Lipschitz quaternions modulo $n$. This article is motivated by \cite{J17}, where Grau et al. studied the directed $\Gamma(\mathbb{Z}_{n}[i, j, k])$ and the undirected $\overline{\Gamma}(\mathbb{Z}_{n}[i, j, k])$. They showed some results on the number of vertices, the diameter, the girth and the domination number of these graphs. Unfortunately, they left an open problem: For an odd prime number $p$ and a positive integer $s$, what is the domination number of $\overline{\Gamma}(\mathbb{Z}_{p^{s}}[i, j, k])$? In section two, we solve this open problem.

For a graph $\Gamma$, let $V(\Gamma)$ be the vertex set of $\Gamma$. The graph automorphism $f$ of $\Gamma$ is a bijection $f:V(\Gamma)\rightarrow V(\Gamma)$ such that vertices $a$ and $b$ are adjacent if and only if $f(a)$ and $f(b)$ are adjacent. The set of all automorphisms of $\Gamma$ , denoted by $\mathrm{Aut}(\Gamma)$, forms a group under composition of transformations.

Recall that $\mathbb{Z}_{p^{s}}[i,j,k]$ is isomorphic to the full matrix ring $M_{2}(\mathbb{Z}_{p^{s}})$ \cite{V80}. The automorphisms of the zero divisor graph of $M_{2}(\mathbb{Z}_{p^{s}})$ were completely determined in \cite{Z17}. Consequently, the automorphisms of the zero divisor graph of $\mathbb{Z}_{p^{s}}[i,j,k]$ follow from it. Unfortunately,  the automorphisms of the zero divisor graph of $\mathbb{Z}_{2^{s}}[i,j,k]$ are unknown. In section three, we show the automorphism group of $\Gamma(\mathbb{Z}_{2^{s}}[i,j,k])$.

\section{The Domination Number}

A dominating set for a graph $G$ is a subset of vertices $S$ with every vertex not in $S$ is adjacent to at least one vertex of $S$. The domination number is the number of vertices in a minimal dominating set. In \cite{J17}, Grau et al. showed the domination number of $\overline{\Gamma}(\mathbb{Z}_{n}[i, j, k])$ for some cases, but the case $n$ is a power of prime numbers remained open. In this section, this problem is solved in Theorem \ref{t2.4}.

We will denote by $M_{1\times2}(\mathbb{Z}_{p^{s}})$ the set of all $1\times 2$ matrices over $\mathbb{Z}_{p^{s}}$, and by $\alpha^{t}$ the transpose of $\alpha\in M_{1\times2}(\mathbb{Z}_{p^{s}})$. Let us denote by $M^{1}_{1\times2}(\mathbb{Z}_{p^{s}})$ the subset of $M_{1\times2}(\mathbb{Z}_{p^{s}})$ consisting of the vectors whose first unit component is 1, i.e., $M^{1}_{1\times2}(\mathbb{Z}_{p^{s}})=\{(1~a), (b~1)|a\in \mathbb{Z}_{p^{s}},~b\in D(\mathbb{Z}_{p^{s}}) \}$.

\begin{lemma}\label{l2.1}{\rm\cite[Lemma 2.3]{Z17}}
	If $\alpha\in M^{1}_{1\times2}(\mathbb{Z}_{p^{s}})$, then there exists a unique $\beta\in M^{1}_{1\times2}(\mathbb{Z}_{p^{s}})$ such that $\alpha\beta^{t}=0$.
\end{lemma}

By \cite[Theorem II.9]{N72}, we know the following lemma.

\begin{lemma}\label{l3.1}
	Every matrix $A$ in $Z^{*}(M_{2}(\mathbb{Z}_{p^{s}}))$ is equivalent to $$
	\begin{pmatrix}
	p^{i} & 0 \\
	0 & p^{j}
	\end{pmatrix},
	$$ where $i,j\in\{0,1,\cdots,s-1\}$, if $p^{j}\neq0$, then $j\neq 0$, $i\leqslant j$. The parameters $(i,j)$ is uniquely determined by $A$.
\end{lemma}

Note that in Lemma \ref{l3.1} for $A,B\in M_{2}(\mathbb{Z}_{p^{s}})$, $A$ is equivalent to $B$ means that there exist $P,Q\in U(M_{2}(\mathbb{Z}_{p^{s}}))$ such that $A=PBQ$. Let $E_{ij}$ denote the matrix in $M_{2}(\mathbb{Z}_{p^{s}})$ having $1$ in its $(i,j)$ entry and zeros elsewhere.

\begin{lemma}\label{l3.2}
	Let $A\in Z^{*}(M_{2}(\mathbb{Z}_{p^{s}}))$. Then $A$ has a unique factorization: $$A=u_{1}p^{i}\alpha^{t}\beta+u_{2}p^{j}E_{mn},$$ where $u_{1},u_{2}\in U(\mathbb{Z}_{p^{s}})$, $\alpha,\beta\in M_{1\times2}^{1}(\mathbb{Z}_{p^{s}})$, $m,n\in \{1,2\}$, if $p^{j}\neq0$, then $j\neq 0$, $i\leqslant j$. $(u_{1},u_{2},\alpha,\beta,i,j,m,n)$ are uniquely determined by $A$.
\end{lemma}
\begin{proof}
	By Lemma \ref{l3.1}, suppose that $A$ is equivalent to $	\begin{pmatrix} p^{i} & 0 \\ 0 & p^{j} \end{pmatrix}$, where $i,j\in\{0,1, \cdots, s-1\}$, if $p^{j}\neq0$, then $j\neq 0$, $i\leqslant j$.
	
	If $p^{j}=0$, then there exist fixed $P=\begin{pmatrix} a_{1} & * \\ b_{1} & * \end{pmatrix},Q=\begin{pmatrix}	a_{2} & b_{2} \\ * & *	\end{pmatrix}\in U(M_{2}(\mathbb{Z}_{p^{s}}))$ such that $$A=P\begin{pmatrix} p^{i} & 0 \\	0 & 0 \end{pmatrix}Q=p^{i}P\begin{pmatrix} 1 & 0 \\	0 & 0 \end{pmatrix}\begin{pmatrix}	1 & 0 \\	0 & 0	\end{pmatrix}Q=p^{i}\begin{pmatrix}	a_{1} & 0 \\	b_{1} & 0	\end{pmatrix}\begin{pmatrix}	a_{2} & b_{2} \\	0 & 0	\end{pmatrix}.$$ Without loss of generality we can assume $a_{1},a_{2}\in  U(\mathbb{Z}_{p^{s}})$. Let $u_{1}=a_{1}^{-1}a_{2}^{-1}$, $\alpha=\begin{pmatrix}	1 &	a_{1}^{-1}b_{1} 	\end{pmatrix}$ and $\beta=\begin{pmatrix}	1 & a_{2}^{-1}b_{2} 	\end{pmatrix}$. Thus  $A=u_{1}p^{i}\alpha^{t}\beta$.
	
	If $p^{j}\neq0$, then $j\neq 0$, $i\leqslant j$. Without loss of generality we can assume $A=p^{i}\begin{pmatrix}	a & b \\	c & d	\end{pmatrix}$, where $a\in  U(\mathbb{Z}_{p^{s}})$. Let $P=\begin{pmatrix}	1 & 0 \\	-a^{-1}c & 1	\end{pmatrix}$, $Q=\begin{pmatrix}	1 & -a^{-1}b \\	0 & 1	\end{pmatrix}$. Then $PAQ=ap^{i}\begin{pmatrix}	1 & 0 \\0 & a^{-1}d	\end{pmatrix}$. Since $A$ is  is equivalent to $	\begin{pmatrix} p^{i} & 0 \\ 0 & p^{j} \end{pmatrix}$, we can assume $a^{-1}d=up^{j-i}$ for some $u\in U(\mathbb{Z}_{p^{s}})$. Therefore, $$A=ap^{i}P^{-1}\begin{pmatrix}	1 & 0 \\0 & up^{j-i}	\end{pmatrix}Q^{-1}=ap^{i}P^{-1}\begin{pmatrix}	1 & 0 \\0 & 0	\end{pmatrix}Q^{-1}+aup^{j}E_{22}.$$ Let  $u_{1}=a$, $\alpha=\begin{pmatrix}	1 &	a^{-1}c 	\end{pmatrix}$, $\beta=\begin{pmatrix}	1 & a^{-1}b	\end{pmatrix}$ and $u_{2}=au$. Thus $A=u_{1}p^{i}\alpha^{t}\beta+u_{2}p^{j}E_{22}$, and the lemma follows.
\end{proof}

\begin{theorem}\label{t2.4}
	The domination number of $\overline{\Gamma}(\mathbb{Z}_{p^{s}}[i,j,k])$ is $p+1$, where $p$ is an odd prime number and $s\geq 1$.
\end{theorem}
\begin{proof}
	We use the fact that $\mathbb{Z}_{p^{s}}[i,j,k]\cong M_{2}(\mathbb{Z}_{p^{s}})$. Let $D=\{p^{s-1}\alpha^{t}\alpha\mid\alpha\in M_{1\times2}^{1}(\mathbb{Z}_{p^{s}})\}$. We claim that $D$ is a dominating set for $\overline{\Gamma}(\mathbb{Z}_{p^{s}}[i,j,k])$. Let $A\in V(\overline{\Gamma}(\mathbb{Z}_{p^{s}}[i,j,k]))$. By Lemma \ref{l3.2}, suppose that  $A=u_{1}p^{i}\alpha^{t}\beta+u_{2}p^{j}E_{mn}$. Then there exists a $\gamma\in M_{1\times2}^{1}(\mathbb{Z}_{p^{s}})$ such that $\gamma\alpha^{t}=0$. Set $B=p^{s-1}\gamma^{t}\gamma\in D$. Thus $BA=0$, and D is a dominating set as claimed.
	
    Our next goal is to show that $D$ is a minimal dominating set for $\overline{\Gamma}(\mathbb{Z}_{p^{s}}[i,j,k])$. Suppose that $E$ is a dominating set for $\overline{\Gamma}(\mathbb{Z}_{p^{s}}[i,j,k])$ and $|D|>|E|=e$. It is easy to check that $\mathrm{ann}(p^{s-1}\alpha^{t}\beta)\supseteq \mathrm{ann}(p^{i}\alpha^{t}\beta)\supseteq \mathrm{ann}(\alpha^{t}\beta)\supseteq \mathrm{ann}(\alpha^{t}\beta+up^{j}E_{mn})$, where $i\leqslant s-1$ and $j\neq 0$. Without loss of generality we can assume $E=\{p^{s-1}\alpha^{t}_{1}\beta_{1},p^{s-1}\alpha^{t}_{2}\beta_{2},\cdots,p^{s-1}\alpha^{t}_{e}\beta_{e}\}$. Then there exist $$p^{s-1}\alpha^{t}_{e+1}\alpha_{e+1}\in D\setminus\{p^{s-1}\alpha^{t}_{1}\alpha_{1}, p^{s-1}\alpha^{t}_{2}\alpha_{2}, \cdots,p^{s-1}\alpha^{t}_{e}\alpha_{e}\},$$  $$p^{s-1}\beta^{t}_{e+1}\beta_{e+1}\in D\setminus\{p^{s-1}\beta^{t}_{1}\beta_{1},p^{s-1}\beta^{t}_{2}\beta_{2},\cdots,p^{s-1}\beta^{t}_{e}\beta_{e}\}.$$ Thus, by Lemma \ref{l2.1}, there exist $\gamma_{1},\gamma_{2}\in M_{1\times2}^{1}(\mathbb{Z}_{p^{s}})$ such that for $i\in\{1,2,\cdots,e\}$, $\gamma_{1}\alpha_{i}\neq 0$, $\gamma_{2}\beta_{i}\neq 0$ but $\gamma_{1}\alpha_{e+1}=0$, $\gamma_{2}\beta_{e+1}=0$. Set $B=\gamma_{2}^{t}\gamma_{1}$. Since for any $A\in E$, $AB\neq 0$ and $BA\neq 0$, it follows that $E$ is not a dominating set. An easy computation shows that $\lvert D\rvert=p+1$, which completes the proof.
\end{proof}

Similar to \cite[Theorem 11 and 12]{J17}, we infer the following theorem.

\begin{theorem}
	Let $n=2^{s_{0}}p_{1}^{s_{1}}\cdots p_{m}^{s_{m}}$ with $p_{l}$ is prime and $s_{l}\geq 1$ for every $l$. Then, the domination number of $\overline{\Gamma}(\mathbb{Z}_{n}[i,j,k])$ is $1+m+p_{1}+\cdots +p_{m}$.
\end{theorem}

\section{Automorphism Group of $\Gamma(\mathbb{Z}_{2^{s}}[i,j,k])$}

We know that $\mathbb{Z}_{2^{s}}[i,j,k]$ is reversible, by \cite[Proposition 6]{J17}, this means that for every $\alpha,\beta\in \mathbb{Z}_{2^{s}}[i,j,k]$, $\alpha\beta=0$ implies that $\beta\alpha=0$. Hence, $\Gamma(\mathbb{Z}_{2^{s}}[i,j,k])=\overline{\Gamma}(\mathbb{Z}_{2^{s}}[i,j,k])$. In order to obtain the automorphisms of the zero divisor graph of $\mathbb{Z}_{2^{s}}[i,j,k]$, we give some notation. For a direct zero divisor graph $\Gamma$ and $a,b\in V(\Gamma)$, we denote by $N_{r}(a)$ (resp. $N_{l}(a)$) the set of vertices $b\in V(\Gamma)$ for which $ab=0$ (resp. $ba=0$). We write $N(a)=N(b)$ to denote that $N_{r}(a)=N_{r}(b)$ and $N_{l}(a)=N_{l}(b)$ for $a,b\in V(\Gamma)$. If $N(a)=N(b)$, then we will say that $a,b$ are twin points. If a bijection $\varphi$ on $\Gamma$ acts in such a way that $\varphi(a)=b$ implies that $a,b$ are twin points, then it is an automorphism of $\Gamma$, which is called a regular automorphism of $\Gamma$. The set of all regular automorphisms of $\Gamma$ is denoted by $\mathrm{Reg}(\Gamma)$. for a ring $R$ and $\alpha,\beta\in R$, if there exists a $\gamma\in U(R)$ such that $\alpha\gamma=\beta$, then we write $\alpha\sim\beta$. Obviously, relation $\sim$ is an equivalence relation on $R$. Let $[\alpha]$ be the equivalence class of $\alpha$, that is, $[\alpha]=\{\beta\in R\mid\beta\sim\alpha\}$. Let $D=D_{1}\cup D_{2}$, where $D_{1},D_{2} \subset \mathbb{Z}_{2^{s}}[i,j,k]$, $D_{1}=\{1+i,1+j,1+k\}$ and $D_{2}=\{1+i+j+k,1+i+j-k\}$. If $S$ is a subset of $V(\Gamma(\mathbb{Z}_{2^{s}}[i,j,k]))$, then we denote $N(S)$ by the set of all neighbors of the vertices of $S$ in graph $\Gamma(\mathbb{Z}_{2^{s}}[i,j,k])$.

From \cite[Lemma 2.6.5]{G03} and \cite[Lemma 2(a)]{A93}, we can conclude the following lemma.

\begin{lemma}\label{c2.5}
	Let $0\neq\alpha \in \mathbb{Z}_{2^{s}}[i,j,k]$. Then $\alpha$ has a unique factorization:
	$$\alpha = 2^{l}\pi\alpha_{0},$$
	where $0\leqslant l< s$, $\pi\in \{1,1 + i,1 + j,1 + k,(1 + i)(1 + j),(1 + i)(1 - k)\}$ and
	$\alpha_{0}\in U(\mathbb{Z}_{2^{s}}[i,j,k])$.
\end{lemma}

By Lemma \ref{c2.5}, it is easy to obtain the following lemma.

\begin{lemma}\label{l4.1}
    Let $D=\{1+i,1+j,1+k,1+i+j+k,1+i+j-k\}$. Then the set of the equivalence class of  $\mathbb{Z}_{2^{s}}[i,j,k]$ is $\{[2^{l}\alpha]\mid\alpha\in D, 0\leqslant l<s\}$.
\end{lemma}

Note that $2^{s-1}(1+i+j+k)=2^{s-1}(1+i+j-k)$. Then a trivial verification shows the following lemma.

\begin{lemma}\label{l4.2}
	Let $D=D_{1}\cup D_{2}$, where $D_{1}=\{1+i,1+j,1+k\}$ and $D_{2}=\{1+i+j+k,1+i+j-k\}$.
	
	(i) Let $1\leqslant m\leqslant s-1$. Then $$N([2^{m}])=\bigcup_{\substack{s-m\leqslant l\leqslant s-1\\   \alpha\in D\cup\{1\}}}[2^{l}\alpha]\setminus [2^{m}].$$
	
	(ii) Let $\alpha\in D_{1}$. Then $N(\alpha)=[2^{s-1}\alpha]\cup[2^{s-1}(1+i+j+k)]$.
	
	(iii) Let $1\leqslant m\leqslant s-1$ and $\alpha\in D_{1}$. Then
	$$N([2^{m}\alpha])=\bigcup_{\substack{s-m\leqslant l\leqslant s-1\\   \beta\in D\cup\{1\}}}[2^{l}\beta]\bigcup_{\beta\in D_{2}\cup\{\alpha\}}[2^{s-1-m}\beta]\setminus [2^{m}\alpha]. $$
	
	(iv) Let $\alpha\in D_{2}$. Then $$N(\alpha)=\bigcup_{\beta\in D}[2^{s-1}\beta]\bigcup_{\beta\in D_{2}\setminus\{\alpha\}}[2^{s-2}\beta].$$
	
	(v) Let $1\leqslant m< s-1$ and $\alpha\in D_{2}$. Then 	$$N([2^{m}\alpha])=\bigcup_{s-m\leqslant l\leqslant s-1}[2^{l}]\bigcup_{\substack{s-1-m\leqslant l\leqslant s-1\\   \beta\in D}}[2^{l}\beta]\bigcup_{\beta\in D_{2}\setminus\{\alpha\}}[2^{s-2-m}\beta]\setminus [2^{m}\alpha]. $$
	
	(vi) Let $\alpha\in D_{2}$. Then $$N([2^{s-1}\alpha])=\bigcup_{\substack{0\leqslant l\leqslant s-1 \\\beta\in D}}[2^{l}\beta]\setminus [2^{s-1}\alpha].$$
\end{lemma}

By Lemma \ref{l4.2}, the proof of the following corollary is straightforward.

\begin{corollary}\label{c4.3}
	Let $\alpha,\beta\in \mathbb{Z}_{2^{s}}[i,j,k]$.
	
	(i) If $s$ is even, then $N(\alpha)=N(\beta)$ if and only if $\alpha\sim\beta$.
	
	(ii) If $s$ is odd, then $N(\alpha)=N(\beta)$ if and only if $\alpha\sim\beta$ or $\alpha,\beta\in\{2^{\frac{s-1}{2}}(1+i),2^{\frac{s-1}{2}}(1+j),2^{\frac{s-1}{2}}(1+k)\}$.
\end{corollary}

For a graph $\Gamma$ and $a\in V(\Gamma)$, let $\bar{a}=\{b\in V(\Gamma)\mid a,b\text{ are twin points}\}$. Let $\Gamma_{E}$ be the compressed graph of $\Gamma$, whose vertices are $\{\bar{a}\mid a\in V(\Gamma)\}$ such that distinct vertices $\bar{a}$ and $\bar{b}$ are adjacent if and only if $a$ and $b$ are adjacent in $\Gamma$. For a vertex $\bar{\alpha}$ in the compressed graph of the zero divisor graph $ \Gamma_{E}(\mathbb{Z}_{2^{s}}[i,j,k])$, we denote $N_{E}(\bar{\alpha})$ be the set of all neighbors of $\bar{\alpha}$ in $ \Gamma_{E}(\mathbb{Z}_{2^{s}}[i,j,k])$. By Lemma \ref{l4.2} and Corollary \ref{c4.3}, it is easy to obtain the next proposition.

\begin{proposition}\label{p4.4}
	Let $D=D_{1}\cup D_{2}$, where $D_{1}=\{1+i,1+j,1+k\}$ and $D_{2}=\{1+i+j+k,1+i+j-k\}$.
	
	(i) If $s$ is even, $\alpha\in D_{1}$, $\beta\in D_{2}$, then $$|N_{E}(\overline{2^{m}})|=\begin{cases}
	6m-1,\quad &1\leqslant m<\frac{s}{2}\\
	6m-2,\quad &\frac{s}{2}\leqslant m\leqslant s-1,
	\end{cases}$$
	$$|N_{E}(\overline{2^{m}\alpha})|=\begin{cases}
	6m+2,\quad &0\leqslant m<\frac{s-1}{2}\\
	6m+1,\quad &\frac{s-1}{2}\leqslant m\leqslant s-1,
	\end{cases}$$
	$$|N_{E}(\overline{2^{m}\beta})|=\begin{cases}
	6m+5,\quad &0\leqslant m<\frac{s-1}{2}\\
	6m+4,\quad &\frac{s-1}{2}\leqslant m< s-1.
	\end{cases}$$
	
	(ii) If $s$ is odd, $\alpha\in D_{1}$, $\beta\in D_{2}$, then $$|N_{E}(\overline{2^{m}})|=\begin{cases}
	6m-1,\quad &1\leqslant m<\frac{s}{2}\\
	6m-4,\quad &\frac{s}{2}\leqslant m\leqslant s-1,
	\end{cases}$$
	$$|N_{E}(\overline{2^{m}\alpha})|=\begin{cases}
    6m+2,\quad &0\leqslant m<\frac{s-1}{2}\\
    6m+1,\quad &m=\frac{s-1}{2}\\
    6m-1,\quad &\frac{s-1}{2}< m\leqslant s-1,
    \end{cases}$$
    $$|N_{E}(\overline{2^{m}\beta})|=\begin{cases}
    6m+5,\quad &0\leqslant m<\frac{s-1}{2}\\
    6m+2,\quad &\frac{s-1}{2}\leqslant m< s-1.
    \end{cases}$$
\end{proposition}

Any graph automorphism preserves adjacency. Let $a,b\in\Gamma$, $f\in\mathrm{Aut}(\Gamma)$ such that $f(a)=b$. Then $|N(a)|=|N(b)|$. Therefore, the next lemma is easy to check by Proposition \ref{p4.4}.

\begin{lemma}\label{l4.5}
	Let $f\in \mathrm{Aut}( \Gamma_{E}(\mathbb{Z}_{2^{s}}[i,j,k]))$, $P_{m}=\{\overline{2^{m}(1+i)},\overline{2^{m}(1+j)},\overline{2^{m}(1+k)}\}$, $Q_{n}=\{\overline{2^{n}},\overline{2^{n-1}(1+i+j+k)},\overline{2^{n-1}(1+i+j-k)}\}$, where $0\leqslant m\leqslant s-1$, $1\leqslant n\leqslant s-1$. Then $f$ stabilizes $P_{m}$ and $Q_{n}$. In particular, $f(\overline{2^{s-1}(1+i+j+k)})=\overline{2^{s-1}(1+i+j+k)}$; $f(\overline{2^{\frac{s}{2}}})=\overline{2^{\frac{s}{2}}}$, $f(Q_{\frac{s}{2}}\setminus\{\overline{2^{\frac{s}{2}}}\})=Q_{\frac{s}{2}}\setminus\{\overline{2^{\frac{s}{2}}}\}$, if $s$ is even; $|P_{\frac{s-1}{2}}|=1$, if $s$ is odd.
\end{lemma}

Recall that a graph automorphism $f$ of a graph $\Gamma$ is a bijection $f:V(\Gamma)\rightarrow V(\Gamma)$ such that vertices $a$ and $b$ are adjacent if and only if $f(a)$ and $f(b)$ are adjacent. If $|V(\Gamma)|=n$, then it is clear that $\mathrm{Aut}(\Gamma)$ is isomorphic to a subgroup of $S_{n}$. By Lemma \ref{l4.5}, each graph automorphism of the compressed graph $\Gamma_{E}(\mathbb{Z}_{2^{s}}[i,j,k])$ is determined by its action on $P_{m}$ and $Q_{n}$, where $m\in\{0,\cdots,s-1\}$, $n\in\{1,\cdots,s-1\}$. Therefore, if $s$ is even, then the group of graph automorphism of $\Gamma_{E}(\mathbb{Z}_{2^{s}}[i,j,k])$ is isomorphic to a subgroup of $S_{3}^{2s-2}\times S_{2}\times S_{1}^{2}$. If $s$ is odd, then it is isomorphic to a subgroup of $S_{3}^{2s-2}\times S_{1}^{2}$.

We next define some actions of $S_{3}$ on $P_{m}$ and $Q_{n}$. Let $\alpha\in\{1+i,1+j,1+k\}$ and  $\beta\in\{2,1+i+j+k,1+i+j-k\}$. The action of $S_{3}$ on $P_{m}$ is given by $(g,\overline{2^{m}\alpha})\mapsto\overline{2^{m}g(\alpha})$. The action of $S_{3}$ on $Q_{n}$ is given by $(h,\overline{2^{n-1}\beta})\mapsto\overline{2^{n-1}h(\beta})$. For $h,h'\in S_{3}$ with $h'(\overline{2^{n}})=\overline{2^{n}}$, $h'(\overline{2^{n-1}}(1+i+j+k))=\overline{2^{n-1}(1+i+j-k))}$ and $h'(\overline{2^{n-1}(1+i+j-k))})=\overline{2^{n-1}(1+i+j+k))}$, we define $h^{*}=h'hh'$.

\begin{lemma}\label{l4.6}
	Let $f\in \mathrm{Aut}( \Gamma_{E}(\mathbb{Z}_{2^{s}}[i,j,k]))$. Then $f|_{P_{m}}=f|_{P_{s-1-m}}$, $(f|_{Q_{n}})^{*}=f|_{Q_{s-n}}$. In particular, $f|_{Q_{\frac{s}{2}}}=h_{\frac{s}{2}}'\text{ or }e$, if $s$ is even; $f|_{P_{\frac{s-1}{2}}}=e$, if $s$ is odd.
\end{lemma}
\begin{proof}
	Let $f\in \mathrm{Aut}( \Gamma_{E}(\mathbb{Z}_{2^{s}}[i,j,k]))$. By Lemma \ref{l4.5}, $f|_{P_{m}},f|_{P_{s-1-m}}$ can be viewed in the natural way as an element $g_{m},g_{s-1-m}$ of $S_{3}$. Let $\alpha,\beta\in \{1+i,1+j,1+k\}$. Suppose that $g_{m}(\overline{2^{m}\alpha})=\overline{2^{m}g_{m}(\alpha)}=\overline{2^{m}\beta}$. Since graph automorphism $f$ preserves adjacency, $f(N_{E}(\overline{2^{m}\alpha}))=N_{E}(\overline{2^{m}\beta})$. From Lemma \ref{l4.2} (ii) and (iii), we conclude that $f(\{\overline{2^{s-1-m}\alpha}\}\cup\star)=\{\overline{2^{s-1-m}\beta}\}\cup\star$. Then $f(\overline{2^{s-1-m}\alpha)}=\overline{2^{s-1-m}\beta}$, by Lemma \ref{l4.5}, which means that $g_{s-1-m}(\overline{2^{s-1-m}\alpha})=\overline{2^{s-1-m}g_{s-1-m}(\alpha)}=\overline{2^{s-1-m}\beta}$. Therefore, $f|_{P_{m}}=f|_{P_{s-1-m}}$. Similarly, we infer that $(f|_{Q_{n}})^{*}=f|_{Q_{s-n}}$.
\end{proof}

\begin{theorem}\label{t4.7}
	Let $s\geqslant 1$ be a positive integer, $H=\{(h,h^{*})\mid h\in S_{3}\}$. Then $\mathrm{Aut}( \Gamma_{E}(\mathbb{Z}_{2^{s}}[i,j,k]))$ is a finite direct product of symmetric groups. Specifically, $$\mathrm{Aut}( \Gamma_{E}(\mathbb{Z}_{2^{s}}[i,j,k]))\cong\begin{cases}
	S_{3}^{\frac{s}{2}}\times H^{\frac{s}{2}-1}\times S_{2}\times S_{1}^{2} ,\quad &\text{ s is even,}   \\
	S_{3}^{\frac{s-1}{2}}\times H^{\frac{s-1}{2}}\times S_{1}^{2} ,\quad &\text{ s is odd.}
	\end{cases}$$
\end{theorem}
\begin{proof}
	Let $H=\{(h,h^{*})\mid h\in S_{3}\}$. In fact, by Lemma \ref{l4.6}, the action of $S_{3}\times S_{3}$ on $P_{m}\times P_{s-1-m}$ can be viewed as the action of $S_{3}$ on $P_{m}\times P_{s-1-m}$ which is given by $(g,(\overline{2^{m}\alpha},\overline{2^{s-1-m}\beta}))\mapsto(\overline{2^{m}g(\alpha)},\overline{2^{s-1-m}g(\beta)})$. Similarly, the action of $S_{3}\times S_{3}$ on $Q_{n}\times Q_{s-n}$ can be viewed as the action of $H$ on $Q_{n}\times Q_{s-n}$. Let $s$ is even. By Lemma \ref{l4.2}, a trivial verification shows that if $f\in S_{3}^{\frac{s}{2}}\times H^{\frac{s}{2}-1}\times S_{2}\times S_{1}^{2}$, then $f\in \mathrm{Aut}( \Gamma_{E}(\mathbb{Z}_{2^{s}}[i,j,k]))$. Conversely, suppose that $f\in \mathrm{Aut}( \Gamma_{E}(\mathbb{Z}_{2^{s}}[i,j,k]))$. Then $f$ can be viewed as $$((f|_{P_{0}},f|_{P_{s-1}}),\cdots,(f|_{P_{\frac{s}{2}-1}},f|_{P_{\frac{s}{2}}}),(f|_{Q_{1}},f|_{Q_{s-1}}),\dots,(f|_{Q_{\frac{s}{2}-1}},f|_{Q_{\frac{s}{2}+1}}),k,e,e),$$ where $k\in S_{2}$ is depending on $f|_{Q_{\frac{s}{2}}}$, since Lemma \ref{l4.5}. Therefore, by Lemma \ref{l4.6}, $f\in S_{3}^{\frac{s}{2}}\times H^{\frac{s}{2}-1}\times S_{2}\times S_{1}^{2}$, which is the desired conclusion. Similar considerations apply to the case when $s$ is odd, which completes the proof.
\end{proof}

\begin{lemma}\label{l4.8}
	$\mathrm{Aut}(\Gamma)\slash \mathrm{Reg}(\Gamma)\cong \mathrm{Aut}(\Gamma_{E})$.
\end{lemma}
\begin{proof}
	For $\varphi\in\mathrm{Aut}(\Gamma)$, we define $\varphi '$ from $V(\Gamma_{E})$ to itself by $\varphi '(\overline{a})=\overline{\varphi(a)}$ for any $a\in V(\Gamma_{E})$. Let $f$ be a mapping from $\mathrm{Aut}(\Gamma)$ to $ \mathrm{Aut}(\Gamma_{E}) $ which is given by $f(\varphi)=\varphi '$ for any $\varphi\in \mathrm{Aut}(\Gamma)$. Similar to \cite[Theorem 2.7]{W15}, the lemma follows.
\end{proof}

By Theorem \ref{t4.7}, Lemma \ref{l4.8} and \cite[Lemma 2.3]{W15}, the following theorem is obvious.

\begin{theorem}
	$\mathrm{Aut}( \Gamma(\mathbb{Z}_{2^{s}}[i,j,k]))\cong \mathrm{Reg}( \Gamma(\mathbb{Z}_{2^{s}}[i,j,k])) \rtimes\mathrm{Aut}( \Gamma_{E}(\mathbb{Z}_{2^{s}}[i,j,k]))$.
\end{theorem}

\vskip 30pt
\allowdisplaybreaks{

}
\end{document}